\documentclass[reqno,12pt]{amsart}%

\usepackage{amsfonts}
\usepackage{amsmath}
\usepackage{amssymb}

\newtheorem{theorem}{Theorem}
\theoremstyle{plain}
\newtheorem{corollary}[theorem]{Corollary}
\newtheorem{definition}[theorem]{Definition}
\newtheorem{example}[theorem]{Example}

\newtheorem{lemma}[theorem]{Lemma}
\newtheorem{remark}[theorem]{Remark}
\numberwithin{equation}{section}
 \numberwithin{theorem}{section}

\renewcommand{\thetheorem}{\thesection.\arabic{theorem}}

\def\Kw{\mathbf{w}}
\def\ds{\displaystyle}

\begin{document}

\today

\title[Parameter Estimation in SPDEs With fBM]
{Asymptotic Properties of the Maximum Likelihood Estimator
for Stochastic Parabolic Equations with Additive Fractional Brownian
Motion}

\author{Igor {Cialenco} }
\address{Department of Applied Mathematics,
Illinois Institute of Technology\\
10 West 32nd Str, Bld E1, Room 208, Chicago, IL 60616, USA}
\email{igor@math.iit.edu, http://math.iit.edu/$\sim$igor}

\author{Sergey V. {Lototsky}}  \thanks{SVL acknowledges support
from the NSF CAREER award DMS-0237724, as well as hospitality and
support of the Scuola Normale Superiore (Pisa, Italy).
 SVL and JP are grateful
to the Institut Mittag-Leffler (Djursholm, Sweden)  for
the hospitality and support. The work of JP was also partially supported by
MSMT Research Plan MSM 4977751301}
\address{Department of Mathematics,
University of Southern California\\
3620 S. Vermont Avenue, KAP 108,
Los Angeles, CA 90089, USA}
\email{lototsky@math.usc.edu, http://math.usc.edu/$\sim$lototsky}

\author{Jan Pospisil}
\address{Department of Mathematics\\
Faculty of Applied Sciences\\
 University of West Bohemia\\
 Univerzitn\'{i}~22, 306~14\\
  Plze\v{n}, Czech Republic}
\email{honik@kma.zcu.cz,  http://www.kma.zcu.cz/Pospisil}

\begin{abstract}
A parameter estimation problem is considered for a diagonaliazable
stochastic evolution equation using a finite number of the Fourier coefficients
of the solution. The equation is driven by additive noise that is white in space and
fractional in time with the Hurst parameter $H\geq 1/2$.
 The objective is to study  asymptotic properties of the maximum likelihood
 estimator as the number of the Fourier coefficients increases.
  A necessary and sufficient condition for consistency and asymptotic  normality
  is presented in terms of the eigenvalues of the operators in the equation.
\end{abstract}

\keywords{Asymptotic normality, ergodicity,
parameter estimation, stochastic evolution equations.}
\subjclass[2000]{Primary 60H15; Secondary 62F12}

\maketitle

\section{Introduction}

In the classical statistical estimation problem, the starting point is a family
$\mathbf{P}^{\theta}$ of probability measures depending on the parameter
$\theta$ in some subset $\Theta$ of a finite-dimensional
Euclidean space. Each $\mathbf{P}^{\theta}$ is the distribution of a random
element. It is assumed that a  realization of one random element
 corresponding to one value $\theta=\theta_0$ of the parameter is
 observed, and the objective is to estimate the values of this
 parameter from the observations.

 The intuition is to select the value $\theta$ corresponding to
 the random element that is {\em most likely} to produce the observations.
 A rigorous mathematical implementation of this idea leads to the notion
 of the regular statistical model \cite{IbragimovKhasminskiiBook1981}:
 the statistical model (or estimation problem)
 $\mathbf{P}^{\theta},\ \theta\in\Theta$, is
 called regular, if the following two conditions are satisfied:
 \begin{itemize}
 \item there exists a probability measure $\mathbf{Q}$ such that
 all measures $\mathbf{P}^{\theta}$ are absolutely continuous
 with respect to $\mathbf{Q}$;
 \item the density $d\mathbf{P}^\theta/d\mathbf{Q}$, called the
 likelihood ratio, has a special property,
 called local asymptotic normality.
 \end{itemize}
 If at least one of the above conditions is violated, the problem is
 called singular.

 In regular models, the estimator $\widehat{\theta}$
  of the unknown parameter is constructed by
 maximizing the likelihood ratio and is called the maximum likelihood estimator
 (MLE). Since, as a rule, $\widehat{\theta}\not=\theta_0$,
 the consistency of the estimator is studied, that is, the convergence
 of $\widehat{\theta}$ to $\theta_0$ as more and more information becomes
 available. In all known regular statistical problems, the amount of information
 can be increased in one of two ways: (a) increasing the sample size,
 for example, the observation time interval (large sample asymptotic);
(b) reducing the amplitude of noise (small noise asymptotic).

In finite-dimensional models, the only way to increase the sample
size is to increase the observation time. In infinite-dimensional models,
in particular, those provided by stochastic partial differential equations
(SPDEs), another possibility is to increase the dimension of the spatial
projection of the observations. Thus, a consistent estimator can be possible
on a finite time interval with fixed noise intensity.
 This possibility was first suggested by
Huebner at al. \cite{HubnerRozovskiiKhasminskii} for parabolic equations driven
by additive space-time  white noise, and was further investigated
by Huebner and Rozovskii \cite{HuebnerRozovskii}, where a necessary and sufficient
condition for the existence of a consistent estimator was stated in terms of the
orders of the operators in the equation.

The objective of the
current paper is to extend the model from \cite{HuebnerRozovskii} to
parabolic equations in which the time component of the noise is fractional
with the Hurst parameter $H\geq1/2$. More specifically, we consider an
abstract evolution equation
\begin{equation}
\label{intr1}
u(t)+\int_0^t(\mathcal{A}_0+\theta \mathcal{A}_1)u(s)ds = W^H(t),
\end{equation}
where $\mathcal{A}_0$, $\mathcal{A}_1$ are known linear operators and
$\theta\in \Theta\subseteq \mathbb{R}$ is the unknown parameter; the
zero initial condition is taken to simplify the presentation.
The noise $W^H(t)$ is a formal series
\begin{equation}
\label{intr0}
W^H(t)=\sum_{j=1}^{\infty} w_j^H(t)h_j,
\end{equation}
where $\{w_j^H,\ j\geq 1\}$ are independent fractional Brownian motions with the
same Hurst parameter $H\geq 1/2$ and $\{h_j,\ j\geq 1\}$ is an orthonormal basis in a
Hilbert space $\mathbf{H}$; $H=1/2$ corresponds to the usual space-time white noise.
 Existence and uniqueness of the solution for such
equations are well-known for all $H\in (0,1)$
 (see, for example, Tindel et al.
\cite[Theorem 1]{TTV-linPDE}).

 The main additional assumption about \eqref{intr1}, both in
 \cite{HuebnerRozovskii} and in the current paper,
is that the equation is {\em diagonalizable}:
$\{h_j,\ j\geq 1\}$ from \eqref{intr0} is a common system of
eigenfunction of the operators $\mathcal{A}_0$ and $\mathcal{A}_1$:
\begin{equation}
\label{intr2}
\mathcal{A}_0h_j=\rho_jh_j,\ \mathcal{A}_1=\nu_jh_j.
\end{equation}
Under certain conditions on the numbers $\rho_j, \ \nu_j$,
 the solution of \eqref{intr1} is a convergent
Fourier series $u(t)=\sum_{j\geq 1} u_j(t)h_j$, and each $u_j(t)$
is a fractional Ornstein-Uhlenbeck (OU)  process. An
$N$-dimensional projection of the solution is then an
$N$-dimensional fractional OU process with independent components.
A  Girsanov-type formula (for example, from Kleptsyna et al.
\cite[Theorem 3]{Klep2}) leads to a maximum likelihood estimator
$\hat{\theta}_N$ of $\theta$ based on the first $N$ Fourier
coefficients $u_1,\ldots, u_N$ of the solution of \eqref{intr1}.
An explicit expression for this estimator exists but requires a
number of additional notations; see formula \eqref{MLE} on page
\pageref{MLE} below.

The following  is the main results of the paper.

\begin{theorem}
Define $\mu_j=\theta\nu_j+\rho_j$ and  assume that the series
$\sum_j (1+|\mu_j|)^{-\gamma}$ converges for some $\gamma>0$.
Then the maximum likelihood estimator $\hat{\theta}_N$
 of $\theta$ is strongly consistent and
asymptotically normal, as $N\to \infty,$ if and only if the
series $\sum_j\nu_j^2\mu_j^{-1}$ diverges; the rate of convergence of the estimator
is given by the square root of the partial sums of this series: as $N\to \infty$,
the sequence
$\left(\sum_{j\leq N} \nu_j^2\mu_j^{-1}\right)^{1/2}(\hat{\theta}_N-\theta)$
converges in distribution to a standard Gaussian random variable.
\end{theorem}

If the operators $\mathcal{A}_0$ and $\mathcal{A}_1$ are elliptic of
orders $m_0$ and $m_1$  on $L_2(M)$, where $M$ is a $d$-dimensional
manifold, and $2m=\max(m_0,m_1)$, then the condition
of the theorem becomes $m_1\geq m-(d/2)$; in the case $H=1/2$ this is known from
\cite{HuebnerRozovskii}. Thus, beside extending the results of
\cite{HuebnerRozovskii} to fractional-in-time noise, we also generalize the
necessary and sufficient condition for consistency of the estimator.

While parameter estimation for the finite-dimensional
fractional OU  and similar processes
has been recently investigated by Tudor and Viens \cite{TV-fBMest}
for all $H\in (0,1)$, our analysis in infinite dimensions
 requires more delicate results:  an explicit
expression for the Laplace transform of a certain functional of the fractional
OU process, as obtained by Kleptsyna and Le Brenton \cite{Klep1}, and
for now this expression exists only for $H\geq 1/2$.

\section{Stochastic Parabolic Equations with Additive FBM}
\label{Section-Settings}
In this section we introduce a diagonalizable stochastic
parabolic equation
depending on a parameter and study the main properties of the
solution.

Let $\mathbf{H}$ be a separable Hilbert space with the inner product
$(\cdot, \cdot)_0$ and the corresponding norm $\|\cdot\|_0$. Let
$\Lambda$ be  a densely-defined linear operator on $\mathbf{H}$ with
the following property: there exists a positive number $c$ such that
 $\| \Lambda u\|_0\geq c\|u\|_0$ for
every $u$ from the domain of $\Lambda$. Then the operator powers
$\Lambda^\gamma, \ \gamma\in\mathbb R,$ are
 well defined and generate the spaces $\mathbf{H}^\gamma$:
 for $\gamma>0$, $\mathbf{H}^\gamma$ is the domain of
  $\Lambda^\gamma$; $\mathbf{H}^0=\mathbf{H}$;
 for $\gamma <0$,  $\mathbf{H}^\gamma$ is the completion
 of $\mathbf{H}$ with respect to the norm
 $\| \cdot \|_\gamma := \|\Lambda^{\gamma} \cdot\|_0$ (see for instance
 Krein at al. \cite{KreinPetuninSemenov}).
 By construction, the collection of spaces
 $\{ \mathbf{H}^\gamma,\ \gamma\in \mathbb R\}$
  has the following properties:
 \begin{itemize}
 \item[$\circ$] $\Lambda^{\gamma}(\mathbf{H}^r) = \mathbf{H}^{r-\gamma}$
  for every
$\gamma,r\in\mathbb{R}$;
 \item[$\circ$] For $\gamma_1<\gamma_2$ the space $\mathbf{H}^{\gamma_2}$
  is densely and
 continuously embedded into $\mathbf{H}^{\gamma_1}$:
 $\mathbf{H}^{\gamma_2}
 \subset \mathbf{H}^{\gamma_1}$ and there exists a positive
 number $c_{12}$ such that  $\|u\|_{\gamma_1}
 \leq c_{12} \|u\|_{\gamma_2}$
 for all $u\in \mathbf{H}^{\gamma_2}$;
\item[$\circ$]  For every $\gamma\in\mathbb R$ and $m>0$, the space
$\mathbf{H}^{\gamma -m}$ is the dual of $\mathbf{H}^{\gamma+m}$ relative to
the inner product in $\mathbf{H}^{\gamma}$, with duality
$\langle\cdot,\cdot\rangle_{\gamma,m}$ given by
$$
\langle u_1, u_2 \rangle_{\gamma,m} = (\Lambda^{\gamma -m}u_1,
\Lambda^{\gamma+m}u_2)_0, \ {\rm where\ }
u_1\in\mathbf{H}^{\gamma-m},\ u_2\in\mathbf{H}^{\gamma +m}.
$$
\end{itemize}

In the above construction, the operator $\Lambda$ can be bounded,
and then the norms in all the spaces $\mathbf{H}^{\gamma}$
will be equivalent. A more interesting situation is therefore
when $\Lambda$ is unbounded and plays the role of the first-order
operator.

Let $(\Omega, \mathcal{F},  \mathbb P)$
 be a probability space and let $\{w_j^H,\ j\geq 1\}$ be a
collection of independent fractional Brownian motions on this space
with the same Hurst parameter $H\in (0,1)$:
$$
\mathbb{E}w_j^H(t)=0,\ \ \mathbb{E}\big(w_j^H(t)w_j^H(s)\big)=
\frac{1}{2}\left(t^{2H}+s^{2H}-|t-s|^{2H}\right).
$$
 Consider the following  equation:
\begin{equation}
\label{eq2}
\begin{cases}
  du(t) + (\mathcal{A}_0 + \theta \mathcal{A}_1) u(t)dt
  =
  \sum\limits_{j\geq 1}g_j(t)dw_j^H(t), \ 0<t\leq T,  \\
  u(0) = u_0\,
\end{cases}
\end{equation}
where $\mathcal{A}_0, \, \mathcal{A}_1$ are linear operators,
 $g_j$ are  non-random,
 and $\theta$ is a scalar parameter  belonging to
an open set $\Theta\subset \mathbb{R}$.

\begin{definition}
\label{def000}
{$     $}

\begin{enumerate}
\item  Equation  \eqref{eq2} is called  {\tt diagonalizable} if the  operators
$\mathcal{A}_0,\ \mathcal{A}_1$, have a common system of eigenfunctions
$\{h_j,\ j\geq 1\}$ such that $\{h_j,\ j\geq 1\}$ is an orthonormal basis in
$\mathbf{H}$ and each $h_j$ belongs to  $\bigcap_{\gamma\in \mathbb{R}}
\mathbf{H}^{\gamma}$.

\item
 Equation  \eqref{eq2}
 is called {\tt $(m,\gamma)$-parabolic} for some
 numbers $m\geq 0$ and $\gamma \in\mathbb{R}$ if
\begin{itemize}
\item[$\circ$]
the operator $\mathcal{A}_0+\theta\mathcal{A}_1$ is uniformly
 bounded from
$\mathbf{H}^{\gamma+m}$ to $\mathbf{H}^{\gamma-m}$ for
$\theta \in \Theta:$ there exists a positive real number $C_1$
such that
\begin{equation}
\label{contA}
\|(\mathcal{A}_0+\theta\mathcal{A}_1)v\|_{\gamma-m}\leq
 C_1\|v\|_{\gamma+m}
\end{equation}
for all $\theta\in \Theta$, $v\in \mathbf{H}^{\gamma+m}$;
\item[$\circ$]
 there exists a positive number $\delta$ and a real number $C$
such that, for every $v\in \mathbf{H}^{\gamma+m}$, $\theta\in \Theta$,
\begin{equation}
\label{parab}
-2\langle (\mathcal{A}_0+\theta\mathcal{A}_1)v,v\rangle_{\gamma,m}
+\delta\|v\|_{\gamma+m}^2\leq
C\|v\|_{\gamma}^2.
\end{equation}
\end{itemize}
\end{enumerate}
\end{definition}

\begin{remark}
\label{rm00}
{\rm
 If equation \eqref{eq2} is $(m,\gamma)$-parabolic,
  then condition \eqref{parab} implies that
$$
\langle (2\mathcal{A}_0+2\theta\mathcal{A}_1+CI)v,v\rangle_{\gamma,m}
\geq \delta\|v\|_{\gamma+m}^2,
$$
where $I$ is the identity operator.
The Cauchy-Schwartz inequality and the continuous embedding of
$\mathbf{H}^{\gamma+m}$ into $\mathbf{H}^{\gamma}$ then imply
$$
\|(2\mathcal{A}_0+2\theta\mathcal{A}_1+CI)v\|_{\gamma}
\geq \delta_1\|v\|_{\gamma}
$$
for some $\delta_1>0$ uniformly in $\theta\in \Theta$.
As a result, we can take $\Lambda=
(2\mathcal{A}_0+2\theta^*\mathcal{A}_1+CI)^{1/(2m)}$ for some fixed
$\theta^*\in\Theta$. If the operator
$\mathcal{A}_0+\theta\mathcal{A}_1$ is unbounded, it is
natural to say that $\mathcal{A}_0+\theta\mathcal{A}_1$
 has order $2m$ and $\Lambda$ has order $1$.
}
\end{remark}

{\em From now on, if equation \eqref{eq2} is
 $(m,\gamma)$-parabolic and
diagonalizable,
we will assume that the operator $\Lambda$ has the same
eigenfunctions as the
operators $\mathcal{A}_0,\ \mathcal{A}_1$;
by Remark \ref{rm00}, this leads to no loss of generality.}

For a diagonalizable equation,  condition
\eqref{parab} can be expressed in terms of the eigenvalues of the
operators in the equation.

\begin{theorem}
\label{th0}
Assume that equation \eqref{eq2} is diagonalizable and
$$
\mathcal{A}_0h_j=\rho_jh_j,\
\mathcal{A}_1h_j=\nu_jh_j.
$$
With no loss of generality {\rm (see Remark \ref{rm00})},
 we also assume that
 $$
 \Lambda h_j=\lambda_jh_j.
 $$
 Then equation \eqref{eq2} is $(m,\gamma)$-parabolic
if and only if there exist  positive real numbers $\delta, C_1$ and
a real number $C_2$ such that, for all $j\geq 1$ and $\theta\in \Theta$,
\begin{align}
&\lambda_j^{-2m}|\rho_j+\theta\nu_j|  \leq C_1; \label{eig1} \\
&-2(\rho_j+\theta\nu_j)+
\delta\lambda_j^{2m} \leq C_2. \label{eig2}
\end{align}

\end{theorem}

\begin{proof}
We show that, for a diagonalizable equation, \eqref{eig1} is equivalent
to \eqref{contA} and \eqref{eig2} is equivalent to \eqref{parab}.
Indeed, note that for every $\gamma,\ r\in \mathbb{R}$,
$$
\|h_j\|_{\gamma+r}=\|\Lambda^rh_j\|_{\gamma}
=\lambda_j^r\|h_j\|_{\gamma}.
$$
Then \eqref{eig1} is \eqref{contA}  and
\eqref{eig2} is  \eqref{parab}, with $v=h_j$.
Since both \eqref{eig1} and
\eqref{eig2} are uniform in $j$ and the
collection $\{h_j,\ j\geq 1\}$ is dense in every $\mathbf{H}^{\gamma}$,
 the proof of the theorem is complete.
\end{proof}

\begin{remark}
{\rm (a) As conditions \eqref{eig1},
\eqref{eig2} do not involve $\gamma$, we
conclude that a diagonalizable  equation is $(m,\gamma)$-parabolic
 {\em for some} $\gamma$ if and only if it is $(m,\gamma)$-parabolic
{\em for every} $\gamma$. As a result, in the future we will simply say
that the equation is $m$-parabolic.

(b) If the operators $\mathcal{A}_0+\theta\mathcal{A}_1$ and
$\Lambda$ are unbounded, then \eqref{eig2} implies that
$\mu_j(\theta)=\rho_j+\theta\nu_j$
 is positive for all sufficiently large $j$.
}
\end{remark}

From now on we will assume that equation \eqref{eq2} is diagonalizable
and fix the basis $\{h_j,\ j\geq 1\}$ in $\mathbf{H}$. Since each $h_j$
belongs to every $\mathbf{H}^{\gamma}$ and, by construction,
 $\bigcap_{\gamma} \mathbf{H}^{\gamma}$ is dense in
 $\bigcup_{\gamma} \mathbf{H}^{\gamma}$, every element $f$ of
$\bigcup_{\gamma}\mathbf{H}^{\gamma}$ has a unique expansion
$\sum_{j\geq 1} f_jh_j$, where $f_j=\langle f,h_j\rangle_{0,m}$ for
a suitable $m$.

\begin{definition}
 The {\tt space-time fractional Brownian motion}
 $W^H$ is an element of $\bigcup_{\gamma\in \mathbb{R}} \mathbf{H}^{\gamma}$
 with the expansion
\begin{equation}
\label{SfBM} W^H(t)=\sum_{j\geq 1} w_j^H(t)h_j.
\end{equation}
\end{definition}

\begin{definition}
\label{def:sol}
Let $W^H$ be a space-time fractional Brownian motion.
The solution of the diagonalizable equation
\begin{equation}
\label{eq200}
\begin{cases}
  du(t) + (\mathcal{A}_0 + \theta \mathcal{A}_1) u(t)dt
  =
  dW^H(t), \ 0<t\leq T,  \\
  u(0) = u_0\,
\end{cases}
\end{equation}
 $u_0\in \mathbf{H},$ is a random process with
values in $\bigcup_{\gamma}\mathbf{H}^{\gamma}$ and an
expansion
\begin{equation}
\label{eq-sol}
u(t)=\sum_{j\geq 1} u_j(t)h_j,
\end{equation}
where
\begin{equation}
\label{eq:OUk}
u_j(t)=(u_{0},h_j)_0\,e^{-(\theta\nu_j+\rho_j)t}
+
\int_0^te^{-(\theta\nu_j+\rho_j)(t-s)}dw^H_j(s).
\end{equation}
\end{definition}

Notice that, due to the special structure of the equation,
 Definition \ref{def:sol} implies both
existence and uniqueness of the solution.

To simplify further notations we write
\begin{equation}
\label{muk}
\mu_j(\theta)=\theta\nu_j+\rho_j.
\end{equation}
By \eqref{eig2}, if equation \eqref{eq2} is $m$-parabolic and diagonalizable, then,
for every $\theta\in \Theta$,
there exists a positive integer $J$ such that
$$
\mu_j(\theta)>0\ \ \  {\rm for \ all\ } \ j\geq J.
$$

\begin{theorem}
\label{th-exist}
Assume that
\begin{enumerate}
\item $H\geq 1/2$;
\item equation \eqref{eq2} is $m$-parabolic and diagonalizable;
\item There exists a positive real number $\gamma$ such that
\begin{equation}
\label{eq:unbnd}
\sum_{j\geq 1} (1+|\mu_j(\theta)|)^{-\gamma}<\infty.
\end{equation}
\end{enumerate}
Then, for every $t>0$,
\begin{enumerate}
\item $W^H(t)\in L_2(\Omega; \mathbf{H}^{-m\gamma})$;
\item $u(t)\in L_2(\Omega; \mathbf{H}^{-m\gamma+2mH})$.
\end{enumerate}
\end{theorem}

\begin{proof}
Condition \eqref{eq:unbnd} implies that $\lim_{j\to \infty} |\mu_j| = \infty$,
 and consequently the operators
$\mathcal{A}_0+\theta\,\mathcal{A}_1$ and $\Lambda$ are unbounded.
The parabolicity assumption and Theorem \ref{th0} then imply that,
for all sufficiently large $j$,
$$
1+|\mu_j(\theta)|\leq C_2\lambda_j^{2m},
$$
 uniformly in $\theta\in \Theta$.
$$
\mathbb{E}\|W^H(t)\|^2_{-m\gamma}=
t^{2H}\sum_{j\geq 1} \lambda_j^{-2m\gamma}\leq C_2t^{2H}\sum_{j\geq 1}
(1+|\mu_j(\theta)|)^{-\gamma}<\infty.
$$

Next, the properties of the fractional Brownian motion imply
$$
\mathbb{E}u_j^2(t)=H(2H-1)e^{-2\mu_j(\theta)t}
\int_0^t\int_0^t e^{\mu_j(\theta)(s_1+s_2)}|s_1-s_2|^{2H-2}ds_1ds_2;
$$
see, for example, Pipiras and Taqqu \cite[formulas (4.1), (4.2)]{PipTaq-PTRF}.
By direct computation,
\begin{equation}
\label{gamma}
\lim_{j\to \infty} |\mu_j(\theta)|^{2H}\mathbb{E}u_j^2(t)
=H(2H-1)\int_0^{\infty}
x^{2H-2}e^{-x}dx=H(2H-1)\Gamma(2H-1).
\end{equation}
Consequently,
\begin{equation}
\label{eq44}
\sum_{j=1}^{\infty} (1+|\mu_j(\theta)|)^{-\gamma+2H}
 \mathbb{E}|u_j(t)|^2<\infty,
\end{equation}
and the second conclusion of the theorem follows.
\end{proof}

\begin{example}
\label{ex:main}
{\rm
(a) For $0<t\leq T$ and $x\in (0,1)$, consider the equation
\begin{equation}
\label{ex00}
du(t,x)-\theta\,u_{xx}(t,x)dt= dW^H(t,x)
\end{equation}
with periodic boundary conditions, where
$u_{xx}=\partial^2 u/\partial x^2$.
Then $\mathbf{H}^{\gamma}$ is the Sobolev space on the
unit circle (see, for example, Shubin \cite[Section I.7]{Shubin}) and
$\Lambda=\sqrt{I-\boldsymbol{\Delta}}$, where $\boldsymbol{\Delta}$ is
the Laplace operator on $(0,1)$ with periodic boundary conditions.
Direct computations show that equation \eqref{ex00} is diagonalizable;
it is $1$-parabolic if and only if $\theta>0$.
Also, $\mu_j=-\theta \pi^2 j^2$, so that,
by Theorem \ref{th-exist}
the solution $u(t)$ of \eqref{ex00} is an element of
$L_2(\Omega; \mathbf{H}^{-\gamma+2H})$ for every $t>0$,
$\gamma>1/2$, and $\theta>0$.

(b) Let $G$ be a smooth bounded domain in
$\mathbb{R}^d$. Let $\boldsymbol{\Delta}$ be the Laplace operator
on $G$ with zero boundary conditions. It is known
 {\rm (for example, from
Shubin \cite{Shubin})}, that
\begin{enumerate}
\item the eigenfunctions $\{h_j,\ j\geq 1\}$
 of $\boldsymbol{\Delta}$
are smooth in $G$ and form an orthonormal basis in $L_2(G)$;
\item the corresponding eigenvalues $\sigma_j,\ j\geq 1$,
can be arranged so that $0<-\sigma_1\leq -\sigma_2\leq \ldots$,
and there
exists a number $c>0$ such that $|\sigma_j|\sim cj^{2/d}$, that is,
\begin{equation}
\label{Lapl-dom}
\lim_{j\to \infty} |\sigma_j|j^{-2/d}=c.
\end{equation}
\end{enumerate}

We take $\mathbf{H}=L_2(G)$,
$\Lambda=\sqrt{I-\boldsymbol{\Delta}}$, where $I$ is the identity
 operator.
 Then $\|\Lambda u\|_0\geq \sqrt{1-\sigma_1}
\|u\|_0$ and the operator $\Lambda$ generates the Hilbert spaces
$\mathbf{H}^{\gamma}$, and,  for every $\gamma\in \mathbb{R}$,
the space
$\mathbf{H}^{\gamma}$ is the closure of the set of smooth compactly
 supported function on $G$ with respect to the norm
$$
\left(\sum_{j\geq 1} (1+j^2)^{\gamma}|\varphi_j|^2\right)^{1/2},\
{\rm \ where\ } \varphi_j=\int_G \varphi(x)h_j(x)dx,
$$
which is an equivalent norm in $\mathbf{H}^{\gamma}$.
Then, for every $\theta\in \mathbb{R}$, the stochastic equation
\begin{equation}
\label{eq:exmain}
du-\big(\boldsymbol{\Delta}u+\theta u\big)dt =dW^H(t,x)
\end{equation}
is diagonalizable and $1$-parabolic.
Indeed, we have $\mathcal{A}_1=I$,
$\mathcal{A}_0=-\boldsymbol{\Delta}$,
 and
$$
-2\langle \mathcal{A}_0v,v\rangle_{\gamma,1}
=-2\|v\|_{\gamma+1}^2+
2\|u\|_{\gamma}^2,
$$
 so that \eqref{parab} holds with $\delta=2$ and $C=2-\theta$.
 Finally, by \eqref{Lapl-dom} we see
that \eqref{eq:unbnd} holds for every $\gamma>d/2$. As a result,
by Theorem \ref{th-exist},
the solution $u(t)$ of \eqref{eq:exmain} is an element of
$L_2(\Omega; \mathbf{H}^{-\gamma+2H})$ for every $t>0$,
$\gamma>d/2$, and $\theta\in \mathbb{R}$.
}
\end{example}

\section{The Maximum Likelihood Estimator and its Properties}

Consider the diagonalizable equation
\begin{equation}
\label{eq-est00}
du(t) + (\mathcal{A}_0 + \theta \mathcal{A}_1) u(t)dt
  =
  dW^H(t)
\end{equation}
with solution $u(t)=\sum_{j\geq 1} u_j(t)h_j$ given by
\eqref{eq:OUk};
for simplicity, we assume that $u(0)=0$.
Suppose that the processes $u_1(t), \ldots, u_N(t)$
can be observed for all $t\in [0,T]$.
The problem is to estimate the parameter $\theta$ using
these observations.

Recall the notation $\mu_j(\theta)=\rho_j+\nu_j\theta$, where
 $\rho_j$ and $\nu_j$ are the eigenvalues of $\mathcal{A}_0$ and
 $\mathcal{A}_1$, respectively.
Then each $u_j$ is a fractional Ornstein-Uhlenbeck process
satisfying
\begin{equation}
\label{eq:OUk1}
du_j(t)=-\mu_j(\theta)u_j(t)dt+dw_j^H(t),\ u_j(0)=0,
\end{equation}
and, because of the independence of
$w_j^H$ for different $j$,
 the processes $u_1,\ldots, u_N$ are (statistically) independent.

Let $\Gamma$ denote the Gamma-function (see \eqref{gamma}).
Following Kleptsyna and Le Brenton \cite{Klep1}, we introduce the
 notations
\begin{align}
& \kappa_H=2H\Gamma\left(\frac{3}{2}-H\right)\Gamma\left(H+\frac{1}{2}
\right),\
k_H(t,s)=\kappa_H^{-1}s^{\frac{1}{2}-H}(t-s)^{\frac{1}{2}-H};
\label{kH}\\
& \lambda_H=\frac{2H\Gamma(3-2H)\Gamma\left(H+\frac{1}{2}\right)}
{\Gamma\left(\frac{3}{2}-H\right)},\
\Kw_H(t)=\lambda^{-1}_Ht^{2-2H}; \label{wH}\\
& M^H_j(t)=\int_0^tk_H(t,s)dw_j^H(s),\
Q_j(t)=\frac{d}{d\Kw_H(t)}\int_0^t k_H(t,s)u_j(s)ds; \label{MQ}\\
& Z_j(t)=\int_0^tk_H(t,s)du_j(s). \label{Z}
\end{align}
By a Girsanov-type formula (see, for example,
Kleptsyna et al. \cite[Theorem 3]{Klep2}), the
measure in the space of continuous, $\mathbb{R}^N$-valued
functions, generated by the process $(u_1, \ldots, u_N)$
 is absolutely continuous with respect to the measure generated
 by the process $(w_1^H,\ldots, w_N^H)$, and the
 density is
 \begin{equation}
 \label{density}
 \exp\left(-\sum_{j=1}^N\mu_j(\theta)\int_0^TQ_j(s)dZ_j(s)-
 \sum_{j=1}^N\frac{|\mu_j(\theta)|^2}{2}
 \int_0^T Q_j^2(s)d\Kw_H(s)\right).
 \end{equation}
Maximizing this density with respect to $\theta$ gives
the Maximum Likelihood Estimator (MLE):
\begin{equation}
\label{MLE}
\widehat{\theta}_N=
-\frac{\ds \sum_{j=1}^N\int_0^T\nu_jQ_j(s)\big(dZ_j(s)
+\rho_jQ_j(s)d\Kw_H(s)\big)}
{\ds \sum_{j=1}^N \int_0^T \nu_j^2Q_j^2(s)d\Kw_H(s)}.
\end{equation}
An important feature  of \eqref{MLE} is that the process $Z_j$ is a
semi-martingale (\cite[Lemma 2.1]{Klep1}),
and so there is no stochastic integration with respect to
fractional Brownian motion:
$\int_0^T\nu_jQ_j(s) dZ_j(s)$ is an It\^{o} integral. Notice that,
when $H=1/2$, we have $k_H=1$, $\Kw_H(s)=s$, $Q_j(s)=Z_j(s)=u_j(s)$,
and \eqref{MLE} becomes
\begin{equation}
\label{MLE1/2}
\widehat{\theta}_N=
-\frac{\ds \sum_{j=1}^N\int_0^T\nu_ju_j(s)\big(du_j(s)
+\rho_ju_j(s)ds\big)}
{\ds \sum_{j=1}^N \int_0^T \nu_j^2u_j^2(s)du_j(s)},
\end{equation}
which is the MLE from \cite{HuebnerRozovskii}.

Let us also emphasize that an  implementation of \eqref{MLE} is
impossible without the knowledge of $H$.

The following is the main result of the paper.
\begin{theorem}
\label{th:conv}
Under the assumptions of Theorem \ref{th-exist},
the following conditions are equivalent:
\begin{align}
&(1) \ \ \ \sum_{j=J}^{\infty}\frac{\nu_j^2}{\mu_j(\theta)}=+\infty;
\label{eq:cond1}\\
&(2) \ \ \ \lim_{N\to \infty} \widehat{\theta}_N=\theta
{\rm\  with \ probability\  one,} \label{eq:cond2}
\end{align}
where $J=\min\{j:\mu_i(\theta)>0\ {\rm for\  all \ } i\geq j\}$.
\end{theorem}

\begin{proof}
Following Kleptsyna and Le Brenton \cite[Equation (4.1)]{Klep1},
we conclude that
\begin{equation}
\label{conv1} \widehat{\theta}_N-\theta=-\frac{\sum_{j=1}^N
\int_0^T\nu_jQ_j(s)dM_j^H(s)}{\sum_{j=1}^N \int_0^T\nu_j^2Q_j^2(s)d\Kw_H(s)}.
\end{equation}
Both the top and the bottom on the right-hand side of
\eqref{conv1} are sums of independent random variables;
moreover, it is known from \cite[page 242]{Klep1} that
\begin{equation}
\label{conv2}
\mathbb{E}\left(\int_0^TQ_j(s)dM_j^H(s)\right)^2
=\mathbb{E}\int_0^TQ_j^2(s)d\Kw_H(s)ds.
\end{equation}
From the expression for the Laplace transform
of $\int_0^TQ_j^2(s)d\Kw_H(s)ds$ (see \cite[Equation (4.2)]{Klep1})
direct computations show that
\begin{equation}
\label{conv3}
\lim_{j\to \infty}\mu_j(\theta)\mathbb{E}\int_0^TQ_j^2(s)d\Kw_H(s)ds
=\frac{T}{2}>0
\end{equation}
and, with $\mathrm{Var}(\xi)$ denoting the variance of the
random variable $\xi$,
\begin{equation}
\label{conv4}
\lim_{j\to \infty}\mu_j^3(\theta)
\mathrm{Var}\left(\int_0^TQ_j^2(s)d\Kw_H(s)ds\right)=
\frac{T}{2}>0;
\end{equation}
a detailed derivation of \eqref{conv3} and \eqref{conv4} is given in the
appendix, Lemmas \ref{lmA1} and \ref{lmA2} respectively.

We now see that if \eqref{eq:cond1} does not hold, then,
by \eqref{conv3},
the series
$$
\sum_{j\geq 1}\int_0^T \nu_j^2Q_j^2(s)d\Kw_H(s)
$$
converges with probability one, which, by \eqref{conv1},
means that \eqref{eq:cond2} cannot hold.

On the other hand, if \eqref{eq:cond1} holds, then
\begin{equation}
\label{conv000}
\sum_{n\geq J}
\frac{\nu_n^2 \mu_n^{-1}}{\Big(\sum\limits_{j=1}^n \nu_j^2\mu_j^{-1} \Big)^2}
 <\infty.
\end{equation}
Indeed, setting $a_n=\nu_n^2 \mu_n^{-1}$ and $A_n=\sum_{j=1}^n a_j$,
we notice that
$$
\sum_{n\geq J} \frac{a_n}{A_n^2} \leq \sum_{n\geq J+1}
\left( \frac{1}{A_n}-\frac{1}{A_{n-1}}\right) = \frac{1}{A_{{}_J}}.
$$
Then the strong law of large numbers, together with the observation
$$
\mathbb{E}\int_0^TQ_j(s)dM_j^H(s)=0,\ j\geq 1,
$$
implies
$$
\lim\limits_{N\to\infty}
\frac{\sum_{j=1}^N \int_0^T\nu_jQ_j(s)dM_j(s)}
{\sum_{j=1}^N \mathbb{E}\int_0^T\nu_j^2Q_j^2(s)d\mathbf{w}_H(s)} = 0 \quad
\mathrm{with\ probability \ one}.
$$
Next, it follows from \eqref{conv000} and \eqref{eq:unbnd} that
\begin{equation}
\label{conv001}
\sum_{n\geq J}
\frac{\nu_n^4 \mu_n^{-3}}{\Big(\sum\limits_{j=J}^n \nu_j^2\mu_j^{-1} \Big)^2}
 <\infty.
\end{equation}
Then another application of the strong law of large numbers
implies that
\begin{equation}
\label{conv5}
\lim_{N\to \infty}
\frac{\sum_{j=1}^N
\int_0^T\nu_j^2Q_j^2(s)d\Kw_H(s)}{\sum_{j=1}^N
\mathbb{E}\sum_{j=1}^N \int_0^T\nu_j^2Q_j^2(s)d\Kw_H(s)}=1
\end{equation}
with probability one, and
 \eqref{eq:cond2} follows.
\end{proof}

\begin{corollary}
Under assumptions of Theorem \ref{th-exist}, if
\eqref{eq:cond1} holds, then
\begin{equation}
\label{normal}
\lim_{N\to \infty} \sqrt{\sum_{j=J}^N\frac{\nu_j^2}{\mu_j(\theta)}}
\ \Big(\widehat{\theta}_N-\theta\Big)=\zeta
\end{equation}
in distribution, where $\zeta$
is a Gaussian random variable with mean zero.
\end{corollary}
\begin{proof}
This follows from \eqref{conv1},
\eqref{conv5}, and the central limit theorem for the
sum of independent random variables.
\end{proof}

 Let us now consider a more general equation
$$
du=(\mathcal{A}_0+\theta\mathcal{A}_1)udt+\mathcal{B}dW^H(t),
$$
where $\mathcal{B}$ is a linear operator.
 If $\mathcal{B}^{-1}$ exists, the equation reduced to
\eqref{eq-est00} by considering
$v=\mathcal{B}^{-1}u$.
If $\mathcal{B}^{-1}$ does not exist, we have two possibilities:
\begin{enumerate}
\item  $(u_0,h_i)_0=0$ for every $i$ such that $\mathcal{B}h_i=0$.
 In this case, $u_i(t)=0$ for all $t>0$, so that we can
 factor out the kernel of $\mathcal{B}$ and reduce the problem to
invertible $\mathcal{B}$.
\item  $(u_0,h_i)_0\not=0$ for some $i$ such that $\mathcal{B}h_i=0$.
In this case,  $u_i(t)=u_i(0)e^{-\rho_it-
\nu_i\theta t}$ and $\theta$ is determined exactly from the
observations of $u_i(t)$:
$$
\theta=\frac{1}{\nu_i(t-s)}\ln\frac{u_i(s)}{u_i(t)}-
\frac{\rho_i}{\nu_i},\ t\not=s.
$$
\end{enumerate}

Let  $\mathcal{A}_0$, $\mathcal{A}_1$ be differential or
pseudo-differential operators, either on a smooth bounded domain
in $\mathbb{R}^d$ or on a smooth compact $d$-dimensional manifold,
and let  $m_0,m_1$, be the orders of $\mathcal{A}_0$,
$\mathcal{A}_1$ respectively,
so that $2m=\max(m_0,m_1)$. Then, under rather general
conditions we have
\begin{equation}
\label{asympt-order}
\lim_{j\to \infty} |\nu_j|j^{m_1/d}=c_1,\ \
\lim_{j\to \infty} \mu_{j}(\theta)j^{2m/d}=c(\theta)
\end{equation}
for some positive numbers $c_1,c(\theta)$;
see, for example, Il'in \cite{Il'in}
 or  Safarov and Vassiliev \cite{SafVas}. In particular, this is the
case for the operators in equations \eqref{ex00} and
\eqref{eq:exmain}.

If \eqref{asympt-order} holds, then condition \eqref{eq:cond1} becomes
\begin{equation}
\label{order}
m_1\geq m-(d/2),
\end{equation}
which, in the case $H=1/2$, was established by
Huebner and Rozovskii \cite{HuebnerRozovskii}.
In particular, \eqref{order} holds for equation
\eqref{ex00} (where $2m=m_1=2$),
 and for equation \eqref{eq:exmain}
if $d\geq 2$ (where $2m=2$, $m_1=0$).

Note that, at least as long as $H\geq 1/2$, conditions
\eqref{eq:cond1} and \eqref{order} do not involve $H$.

The maximum likelihood estimator
\eqref{MLE} has three features that are clearly
attractive: consistency, asymptotic normality, and absence of stochastic
integration with respect to fractional Brownian motion. On the other
hand, actual implementation of \eqref{MLE} is problematic:
when $H>1/2$, computing the processes $Q_j$ and $Z_j$ is certainly nontrivial.
 Estimator \eqref{MLE1/2} is defined for
all $H\geq 1/2$ and contains only the processes $u_j$, but, when  $H>1/2$,
is not an MLE and is even harder
 to implement because of the stochastic integral with respect to $u_j$.

 With or without  condition \eqref{eq:cond1}, a consistent estimator
of $\theta$ is possible  in the large time asymptotic: for every
$j\geq 1$,
\begin{equation}
\label{LT-KB}
\lim_{T\to \infty}
\frac{ \int_0^T\nu_jQ_j(s)\big(dZ_j(s)
+\rho_jQ_j(s)d\Kw_H(s)\big)}
{ \int_0^T \nu_j^2Q_j^2(s)d\Kw_H(s)}=-\theta
\end{equation}
with probability one (\cite[Proposition 2.2]{Klep1}). For $H>1/2$,
implementation of this estimator is essentially equivalent to the
implementation of \eqref{MLE}.

An alternative to \eqref{LT-KB} was suggested by
Maslowski and Posp\'{\i}\v{s}il \cite{Maslowski.Pospisil08} using the ergodic
properties of the OU process.
Let us first illustrate the idea on  a simple example.

If $a>0$ and $w=w(t)$ is a standard one-dimensional Brownian motion,
then the OU process $dX=-aX(t)dt+dw(t)$ is ergodic and its unique
invariant distribution is normal with zero mean and variance $(2a)^{-1}$.
In particular,
\begin{equation}
\label{erg1}
\lim_{T\to \infty} \frac{1}{T}\int_0^T X^2(t)dt=\frac{1}{2a}
\end{equation}
with probability one, and so
\begin{equation}
\label{erg2}
\tilde{a}(T)=\frac{T}{2\int_0^T X^2(t)dt}
\end{equation}
is a consistent estimator of $a$ in the long-time asymptotic. Note that
the maximum likelihood estimator in this case is
\begin{equation}
\label{erg3}
\hat{a}(T)=-\frac{\int_0^TX(t)dX(t)}{\int_0^TX^2(s)ds}
\end{equation}
and is strongly consistent for every $a\in \mathbb{R}$ \cite[Theorem 17.4]{LSh-II}.

Similarly, if $a>0$, then  the fractional OU process
\begin{equation}
\label{erg4}
dX(t)=-aX(t)dt+dw^H(t),\ X(0)=0
\end{equation}
is Gaussian, and, by \eqref{gamma} on page \pageref{gamma},
 converges in distribution, as $t\to \infty$, to the Gaussian random
 variable with zero mean and variance $c(H)a^{-2H}$, where
 \begin{equation}
 \label{erg5}
c(H)=H(2H-1)\Gamma(2H-1);
\end{equation}
notice that, in the limit $H\searrow 1/2$, we recover the result for the
usual OU process.
Further investigation shows that, similar to \eqref{erg1},
$$
\lim_{T\to \infty} \frac{1}{T}\int_0^T X^2(s)ds= \frac{c(H)}{a^{2H}}
$$
(see \cite{Maslowski.Pospisil08}).
As a result, for every $j$ such that $\theta \nu_j+\rho_j>0$, we have
\begin{equation}
\label{erg6}
\lim_{T\to \infty} \frac{1}{T}\int_0^Tu_j^2(t)dt=\frac{c(H)}{(\theta\nu_j+\rho_j)^{2H}}
\end{equation}
with probability one. Under an additional assumption that $\nu_j\not=0$, we get an
estimator of $\theta$
\begin{equation}
\label{erg7}
\tilde{\theta}^{(j)}(T)=
\frac{1}{\nu_j}\left(\frac{c(H)T}{\int_0^Tu_j^2(t)dt}\right)^{\frac{1}{2H}}-
\frac{\rho_j}{\nu_j}.
\end{equation}
This estimator is strongly consistent in the long time asymptotic:
$\lim_{T\to \infty}|\tilde{\theta}^{(j)}(T)-\theta|=0$ with probability one
  (\cite[Theorem 5.2]{Maslowski.Pospisil08}).
While not a maximum likelihood estimator, \eqref{erg7} is easier to
implement computationally than \eqref{MLE}. If, in Theorem \ref{th-exist}
 on page \pageref{th-exist}, we have  $\mathcal{A}_0=0$, $\nu_j>0$, and
 $\gamma<2H$, then
a version of \eqref{erg8} exists using all the Fourier coefficients
$u_j,\ j\geq 1:$
\begin{equation}
\label{erg8}
\tilde{\theta}(T)=
\left(\frac{c(H)T\sum_{j= 1}^{\infty} \nu_j^{-2H}}{\sum_{j= 1}^{\infty}
\int_0^Tu_j^2(t)dt}\right)^{\frac{1}{2H}};
\end{equation}
see \cite[Theorem 5.2]{Maslowski.Pospisil08}.

An interesting open question related to both \eqref{MLE} and \eqref{erg7},
\eqref{erg8} is how to combine estimation of $\theta$ with estimation of $H$.

\def\cprime{$'$} \def\cprime{$'$} \def\cprime{$'$} \def\cprime{$'$}

%\bibliographystyle{plain}

%\bibliography{igorbibliography,fbm}
%This is to use multiple data bases.

\section*{Appendix}
\renewcommand{\theequation}{A.\arabic{equation}}
\setcounter{equation}{0}

\renewcommand{\thetheorem}{A.\arabic{theorem}}
\setcounter{theorem}{0}

Below, we prove equalities \eqref{conv3} and \eqref{conv4}.
\begin{lemma}
\label{lmA1}
For every $\theta\in\Theta$ and $H\in[1/2,1),$
$$
\lim\limits_{j\to\infty} \mu_j(\theta) \mathbb{E}\int_0^T Q_j^2(s)d\mathbf{w}_H (s)
= \frac{T}{2} \, .
$$
\end{lemma}
\begin{proof}
Denote by $\Psi_T^H(a,\mu_j)$ the Laplace transform of
$\int_0^TQ_j^2(s)d\mathbf{w}_H(s)$,
namely
\begin{equation}
\label{KB-Lap}
\Psi_T^H(a,\mu_j(\theta)) = \mathbb{E}
\exp\left\{ -a  \int_0^TQ_j^2(s)d\mathbf{w}_H(s) \right\}, \quad a>0.
\end{equation}
We will use the expression for  $\Psi_T^H$ from \cite[page 242]{Klep1},
 and write it as follows
$$
\Psi_T^H(a, \mu_j ) = \alpha e^{ \frac{ (\mu_j - \alpha)T}{2}}
\left[ \Delta_T^H(\mu_j,\alpha) \right]^{-\frac{1}{2}}
$$
where $\mu_j=\mu_j(\theta), \ \alpha:=\sqrt{\mu_j^2 + 2a}$,
\begin{align*}
\Delta_T^H(\mu_j,\alpha) &
= \frac{\pi\alpha T e^{-\alpha T} (\alpha^2-\mu_j^2)}{4\sin(\pi H)}
I_{-H}\left(\frac{\alpha T}{2}\right)I_{H-1}\left( \frac{\alpha T}{2}\right) \\
    &+ e^{-\alpha T}\left[ \alpha \sinh \left(\frac{\alpha T}{2}\right)
    + \mu_j \cosh\left(\frac{\alpha T}{2}\right)  \right]^2 \ ,
\end{align*}
and $I_p$ is the modified Bessel function of the first kind and order $p$.

Note that
$$
\mathbb{E}\int_0^T Q_j^2(s)d\mathbf{w}_H (s)
= -\frac{\partial \Psi_T^H(a,\mu_j)}{\partial a}\Big{|}_{a=0} \ .
$$
Direct evaluations (for example, using {\tt Mathematica} computer algebra
system) give
$$
 \frac{\partial \Psi_T^H(a,\mu_j)}{\partial a} \Big{|}_{a=0}=
 \frac{ 2+2e^{\mu_j T}(1-\mu_j T) - \mu_j\pi T I_{H-1}\left(\frac{\mu_j T}{2}\right)
I_{-H}\left( \frac{\mu_j T}{2}\right) \csc(H\pi)}{4\mu_j^{2} e^{\mu_j T}},
$$
where $\csc(x)=1/\sin(x)$.
%and using the identity
%$$
%I_{1-H}(x ) \, I_H(x) -
%I_{-H}(x)\, I_{H-1}(x)
% = \frac{2\sin(\pi H)}{\pi x}
%$$
%%(see for instance [KB], Remark 5.1, page 247),
%we continue
%$$
%\frac{\partial \Psi_T^H(0,\mu_j)}{\partial a} =
%\frac{-2 + 2e^{\mu_j T} - \frac{\mu_j \pi T}{\sin(H\pi)}I_{H}I_{1-H}
% + 2\mu_j Te^{\mu_j T} }
%{4 \mu_j^2 e^{\mu_j T}} \ .
%$$
By combining formulas (6.106), (6.155), and (6.162) in \cite{Andrews},
we conclude that, for all $p\in (-1,1),\ p\not=0,$ we have
$I_p(x)\sim e^x/\sqrt{2\pi x}$, $x\to \infty$, that is,
\begin{equation}
\label{A-B}
\lim_{x\to +\infty} \sqrt{2\pi x}\ e^{-x}I_p(x)=1.
\end{equation}
Therefore
$$
\frac{\partial \Psi_T^H(a,\mu_j)}{\partial a}\Big|_{a=0} \sim
\frac{ 2+2e^{\mu_j T}(1-\mu_j T) - e^{\mu_j T} \csc(H\pi)}{4\mu_j^{2} e^{\mu_j T}}
\sim - \frac{T}{2\mu_j}, \ j\to \infty,
$$
$$
\lim_{j\to\infty} \mu_j \frac{\partial \Psi_T^H(a,\mu_j)}{\partial a} \Big|_{a=0}= -\frac{T}{2},
$$
and the lemma is proved.
\end{proof}

\begin{lemma}
\label{lmA2}
For every $\theta\in\Theta$ and $H\in[1/2,1)$
$$
\lim\limits_{j\to\infty} \mu_j^3(\theta) \mathrm{Var}
\left(\int_0^T Q_j^2(s)d\mathbf{w}_H (s) \right)= \frac{T}{2} \, .
$$
\end{lemma}
\begin{proof}
Note that
\begin{equation}
\label{App-Var}
\mathbf{V}:=\mathrm{Var}\left(\int_0^T Q_j^2(s)d\mathbf{w}_H (s)\right) =
\left[\frac{\partial^2 \Psi_T^H(a,\mu_j)}{\partial a^2} -
\left(\frac{\partial \Psi_T^H(a,\mu_j)}{\partial a}\right)^2\right]_{a=0},
\end{equation}
with $\Psi^H_T$ from \eqref{KB-Lap}.
Direct evaluation of the right hand side of \eqref{App-Var}
(for example, using {\tt Mathematica} computer algebra system) gives
\begin{align*}
\mathbf{V}= \frac{1}{8\mu_j^4 e^{2T\mu_j}}
&\Big( 2-8e^{\mu_j T} (1+\mu_j T) + 2 e^{2\mu_j T}(-5+2\mu_j T) \\
& + \pi \mu_j T \csc(\pi H) \big[ -2e^{\mu_j T} \mu_j T I_{1-H}
\left(\frac{\mu_j T}{2}\right)
 I_{H-1}\left(\frac{\mu_j T}{2}\right)\\
&+ I_{-H}\left(\frac{\mu_j T}{2}\right)
\{ 4(-1+e^{\mu_j T}(1+\mu_j T))I_{H-1} \left(\frac{\mu_j T}{2}\right)\\
&-  2 e^{\mu_j T} \mu_j T I_{H}\left(\frac{\mu_j T}{2}\right)
 +\pi\mu_j T I_{H-1}^2\left(\frac{\mu_j T}{2}\right)
  I_{-H}\left(\frac{\mu_j T}{2}\right) \csc(H\pi) \} \big] \Big),
\end{align*}
where $\csc(x)=1/\sin(x)$
and $I_p$ is the modified Bessel function of the first kind and order $p$.

Using \eqref{A-B}, we conclude that
\begin{align*}
\lim\limits_{j\to\infty} \mu_j^3(\theta)\mathbf{V}
= & \lim\limits_{j\to\infty} \mu_j^3
\Big( \frac{-10 + 4\csc(H\pi) + \csc^2(H\pi)}{8\mu_j^4}
+ \frac{1}{4\mu_j^4 e^{2\mu_j T}} \\
&- \frac{\csc(H\pi) + 2 +2 \mu_j T}{2\mu_j^4 e^{\mu_j T}}
+ \frac{T}{2\mu_j^3}  \Big) \\
= & \frac{T}{2}
\end{align*}
and complete the proof of the lemma.
\end{proof}

\end{document}